\documentclass[12pt]{article}

\usepackage{latexsym,amsfonts,amssymb,epsfig,verbatim,}
\usepackage{amsmath,amsthm,amssymb,latexsym,graphics,textcomp}
\usepackage{graphicx}
\usepackage{url}

\newtheorem{question}{Question}
\newtheorem{thm}{Theorem}[section]
\newtheorem{cor}{Corollary}[section]
\newtheorem{lem}[thm]{Lemma}
\newtheorem*{mythm}{Theorem}
\newtheorem*{ex}{Example}

\newcommand{\BZ}{\mathbb Z}

\newcommand{\BR}{\mathbb R}

\DeclareMathOperator{\tr}{tr}
\DeclareMathOperator{\Stab}{Stab}
\DeclareMathOperator{\PSL}{PSL_2}
\DeclareMathOperator{\length}{length}

\DeclareMathOperator{\Hom}{Hom}

\def\T{\mathcal{T}}

\begin{document}

\title{The Sunada construction and the simple length spectrum} 

\author{Rasimate Maungchang}

\maketitle

\begin{abstract}
We show that certain families of iso-length spectral hyperbolic surfaces obtained via the Sunada construction are not generally simple iso-length spectral.
\end{abstract}
\section{Introduction}
\label{sec:Introduction}

Let $M$ be a compact Riemannian manifold. The \textbf{length spectrum} $L(M)$ of $M$ is the set of all lengths of closed geodesics on $M$ counted with multiplicities. Two manifolds $M_1$ and $M_2$ are said to be \textbf{iso-length spectral} if $L(M_1)=L(M_2)$. 

In \cite{Sunada}, Sunada provided a method to construct iso-length spectral manifolds that are frequently not isometric (see also \cite[Ch.11-13]{BuserBook}). This requires a notion from group theory.

Let $G$ be a finite group. Two subgroups $H$ and $K$ of $G$ are said to be \textbf{almost conjugate} if, for any $g\in G$, 
\[
\left|H\cap(g)\right|=\left|K\cap(g)\right|,
\]
where $(g)$ denotes the conjugacy class of $g$ in $G$.

\begin{mythm}[\textbf{Sunada}]
\label{T:Sunada construction}
Let $M_0$ be a closed Riemannian manifold, $G$ a finite group, and $H$ and $K$ almost conjugate subgroups of $G$. If there is a surjective homomorphism from $\pi_1(M_0)$ onto $G$, then the finite covering spaces $M_H$ and $M_K$ of $M_0$ corresponding to the subgroups $H$ and $K$, respectively, are iso-length spectral.
\end{mythm}

When $H$ and $K$ are not conjugate in $G$, the manifolds $M_H$ and $M_K$ can often be shown to be nonisometric. For example, when $M_0$ is a surface, a generic hyperbolic metric on $M_0$ will produce nonisometric $M_H$ and $M_K$; see \cite[Ch.12.7]{BuserBook}.

For surfaces, the simple closed geodesics often carry more topological information. Accordingly,  the \textbf{simple length spectrum} $L^s(M)$ of $M$ is defined to be the set of all lengths of simple closed geodesics on $M$ counted with multiplicities; see \cite{McShane}. Two manifolds $M_1$ and $M_2$ are said to be \textbf{simple iso-length spectral} if $L^s(M_1)=L^s(M_2)$. 

\begin{question}
\label{Q:Question1}
Are there nonisometric simple iso-length spectral hyperbolic surfaces?
\end{question}

In \cite{McShane}, McShane and Parlier give example of pairs of 4-holed spheres with geodesic boundary which have the same \textit{interior simple lengh spectrum} (one ignores the boundary lengths). They do in fact have different boundary lengths, and so they have different simple length spectrum.

One can ask if Sunada's construction provides a positive resolution to Question~\ref{Q:Question1}.

\begin{question}
\label{Q:question2}
Does Sunada's construction, for a given homomorphism \\$\rho:\pi_1(M_0) \to G$, generically give simple iso-length spectral surfaces?
\end{question}

To answer Question~\ref{Q:question2}, we choose one of the examples of almost conjugate subgroups Sunada provided in his paper \cite{Sunada}.

\begin{ex}
$G=(\BZ/8\BZ)^{\times} \ltimes \BZ/8\BZ$ with usual action of $(\BZ/8\BZ)^{\times}$ on $\BZ/8\BZ$.

$H=\left\{(1,0),(3,0),(5,0),(7,0)\right\}$ and
$K=\left\{(1,0),(3,4),(5,4),(7,0)\right\}$ are almost conjugate but not conjugate.
\end{ex}

Our main theorem is the following.

\begin{thm}
\label{T:main theorem}
Let $M_0$ be a closed oriented surface of genus 2, $G$, $H$, and $K$ the groups provided in the example above.

There is a surjective homomorphism $\rho:\pi_{1}(M_0) \to G$ such that, for almost every $[m]\in \T(M_0)$, the corresponding iso-length spectral surfaces $M_H$ and $M_K$ are not simple iso-length spectral.
\end{thm}
In fact, we prove a little bit more. We define the \textbf{length set} and the \textbf{simple length set} of a manifold $M$ to be the set of all lengths of closed geodesics on $M$ without multiplicities and the set of all lengths of simple closed geodesics on $M$ without multiplicities, respectively. Then from the proof of Theorem~\ref{T:main theorem} we have the following corollary.

\begin{cor}
\label{C:cor}
The surfaces $M_H$ and $M_K$ in Theorem~\ref{T:main theorem} have the same length set but they do not have the same simple length set.
\end{cor}

This corollary shows that the construction of length equivalent manifolds in~\cite{Leininger2} does not necessarily give simple length equivalent manifolds.

\bigskip

{\bf Outline of the paper.} Section~\ref{sec:background} contains the relevant background.
In Section~\ref{sec:proof of the main theorem}, we give the proof of the main theorem. The sketch of the proof is as follow. We begin by defining a surjective homomorphism $\rho:\pi_{1}(M_0) \to G$ and a closed curve $\alpha$ in $M_0$. By Sunada's construction, the covering spaces $\pi_H:M_H \to M_0$ and $\pi_K:M_K \to M_0$ corresponding to the subgroups $H$ and $K$ are iso-length spectral. We then show that, for almost every $[m]\in\T(M_0)$, the induced metrics on $M_H$ and $M_K$ have the following property. In each of these two covering spaces $M_H$ and $M_K$, there are exactly four closed geodesics having the same length as $\alpha$, namely the two degree-one components of $\pi_H^{-1}(\alpha)$ (and $\pi_K^{-1}(\alpha)$) and their images under the lifts of the hyperelliptic involution $\tau:M_0 \to M_0$. We also show that these four closed geodesics on $M_H$ are nonsimple while the other four closed geodesics on $M_K$ are simple. Therefore $M_H$ and $M_K$ are not simple iso-length spectral.

We remark on one subtlety of the proof. According to \cite{Randol}, there are curves $\gamma, \gamma'$ on $M_0$ such that for every hyperbolic metric $m$ on $M_0$, $\length_m(\gamma)=\length_m(\gamma')$. Although these are nonsimple on $M_0$, they become simple in a finite sheeted cover, so must be accounted for in our proof.
\section{Background}
\label{sec:background}
Let $M$ be a closed oriented surface of genus $g \geq 2$. We denote the Teichm\"uller space of $M$ by
\[
 \T(M)=\left\{[m]\mid m\:\:\mbox{is a hyperbolic metric on}\;M\right\},
\]
where $[m]$ represents the equivalence class via the equivalence relation $m \sim m'$ if there exists an isometry $f:(M,m) \to (M,m')$ such that $f \simeq id_M $, see e.g. \cite{BuserBook}. 

Given $[m]\in\T(M)$, the holonomy homomorphism 
\[
\rho_m:\pi_1(M) \to \PSL(\BR)
\]
is well defined up to conjugation in $\PSL(\BR)$. This determines an embedding 
\begin{equation}
\label{eq:embedding}
\T(M) \to \Hom(\pi_1(M),\PSL(\BR))/\mbox{conjugation}
\end{equation}
by $[m]\mapsto[\rho_m]$.

Let $\gamma$ be an essential closed curve on $M$. The length function of $\gamma$
\[
\length_{(\cdot)}(\gamma): \T(M) \to \BR_+
\]
is defined as the length of the $m$-geodesic homotopic to $\gamma$.  Using the holonomy homomorphism, one can compute
\begin{equation}
\label{eq:length_function}
\length_{[m]}(\gamma)=2\cosh^{-1}
\left(
\frac
{\left|
\tr(
\rho_m(\gamma)
)
\right|}
{2}
\right).
\end{equation}

The embedding~(\ref{eq:embedding}) makes $\T(M)$ into a real analytic manifold. By~(\ref{eq:length_function}), the length functions are analytic (see e.g. \cite{kerckhoff} or \cite{Abikoff}). Since $\T(M)$ is connected, we then have the following theorem; see \cite{McShane}.

\begin{thm}
\label{T:length_function}
Let $c\in\BR$, $\alpha$ and $\beta$ be closed curves on $M$. The function
\[
f=c\cdot\length_{(\cdot)}(\beta)-\length_{(\cdot)}(\alpha):\T(M) \to \BR
\]
is real analytic, in particular, $f\neq0$ almost everywhere or $f=0$ everywhere.
\end{thm}

Let $\gamma$ and $\gamma'$ be closed curves on $M$. The geometric intersection number of $\gamma$ and $\gamma'$ is defined by 
\[
i(\gamma,\gamma')=\min_{\overline{\gamma},\overline{\gamma}\:'}|\left(\overline{\gamma}\times\overline{\gamma}\:'\right)^{-1}\left(\Delta\right)|,
\]
where $\overline{\gamma}$ and $\overline{\gamma}\:'$ are in the homotopy classes $[\gamma]$ and $[\gamma']$, respectively, $\overline{\gamma}\times\overline{\gamma}\:':S^1 \times S^1 \to M \times M$, and $\Delta \subset M \times M$ is diagonal.

The next theorem provides a tool for dealing with the phenomenon arising from \cite{Randol}.
\begin{thm}
\label{T:geometric_intersection}
Let $\gamma$, $\gamma'$ be closed curves on $M$ and $k\in\BR$. 
 
If \;$\length_m(\gamma)=k\cdot\length_m(\gamma')$, for all $[m]\in \T(M)$, then $i(\gamma,\alpha)=k\cdot i(\gamma',\alpha)$, for all simple closed curves $\alpha$ on $M$.
\end{thm}

\begin{proof}
For $k = 1$, a proof can be found in \cite{Leininger}, for example. The same idea works here, and we sketch it.

Given a simple closed curve $\alpha$, there exists a sequence $\left\{\left[m_n\right]\right\}\subset \T(M)$ such that 
\[
\frac{1}{n}\cdot\length_{\left[m_n\right]}(\eta) \to i(\eta,\alpha),
\]
for all closed curves $\eta$ on $M$.

Now suppose $\length_{[m]}(\gamma)=k\cdot\length_{[m]}(\gamma')$ for all $[m] \in \T(M)$. Then
\[
\frac{1}{n}\cdot\length_{\left[m_n\right]}(\gamma) \to i(\gamma,\alpha)
\]
and
\[
\frac{k}{n}\cdot\length_{\left[m_n\right]}(\gamma') \to k\cdot i(\gamma',\alpha).
\]
So $k\cdot i(\gamma',\alpha)=i(\gamma,\alpha)$.
\end{proof}
The following theorem is shown in \cite{Leininger}.
\begin{thm}
\label{T:homology}
Given $\gamma$ and $\gamma'$ closed curves on $M$, if 
\[
\length_{[m]}(\gamma)=\length_{[m]}(\gamma'),
\]
for all $[m]\in \T(M)$, then $[\gamma]=\pm[\gamma']$ in $H_1(M)$.
\end{thm}
\section{Proof of the main theorem}
\label{sec:proof of the main theorem}

Let $M_0$ be a closed oriented surface of genus 2. We write the fundamental group of $M_0$ as $\pi_{1}(M_0)=\langle a,b,c,d|[a,b][c,d]=1 \rangle$, see Figure~\ref{F:2g with generators}.
\begin{figure}[ht]
\begin{center}
\includegraphics[height=2cm]{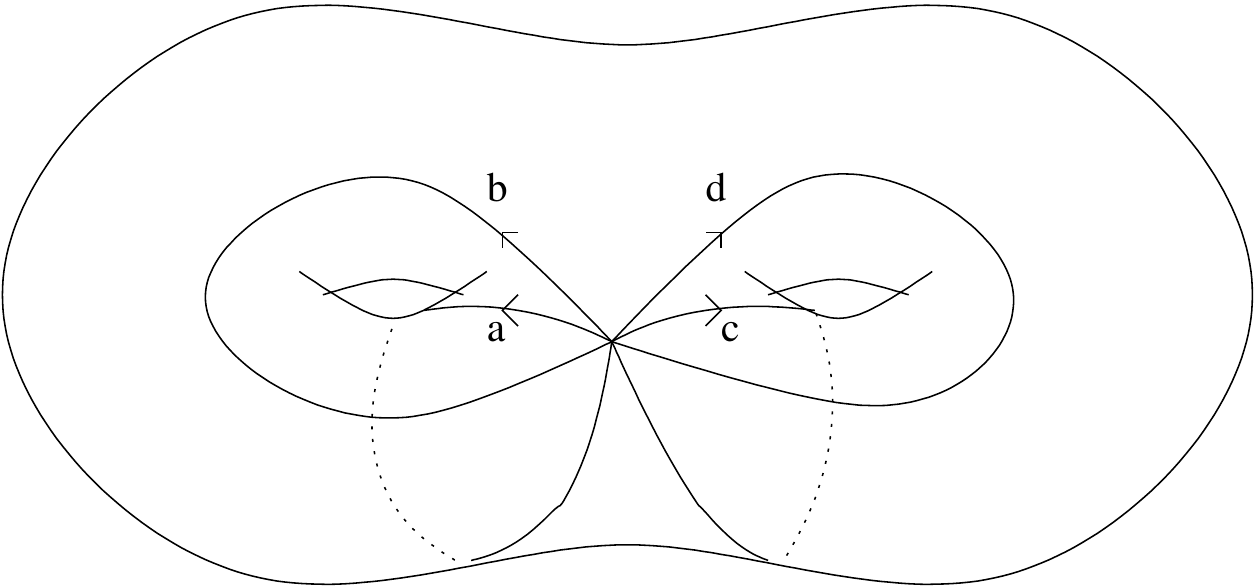} 
\caption{$M_0$ with the generators of $\pi_{1}(M_0)$.} 
\label{F:2g with generators}
\end{center}
\end{figure}

Let $G$, $H$ and $K$ be groups given in the example in Section~\ref{sec:Introduction}. We define a surjective homomorphism $\rho:\pi_{1}(M_0) \to G$ by
\[ 
\rho(a)=(3,0), \quad
\rho(b)=(5,0), \quad
\rho(c)=(1,0), \quad\mbox{and}\quad
\rho(d)=(1,1). 
\]

Let $\pi:M \to M_0$, $\pi_H:M_H \to M_0$ and $\pi_K:M_K \to M_0$ be the covering spaces of $M_0$ corresponding to $\ker(\rho)$, $\rho^{-1}(H)$ and $\rho^{-1}(K)$, respectively. 

To help visualizing the covering space M, first we construct the covering space $\pi:M_N \to M_0$ corresponding to the subgroup $N=\BZ/8\BZ$ of $G$, as shown in Figure~\ref{F:the covering space $M_N$}. Then we construct $M$ from the surjective homomorphism $\sigma:\pi_{1}(M_N) \to N$, the restriction of $\rho$ to $\pi_{1}(M_N) < \pi_{1}(M_0)$, see Figure~\ref{F:the covering space $M$}. Observe that the generator of $\BZ/8\BZ \cong N < G$ translates each piece in Figure~\ref{F:the covering space $M$} to the right, and sends the last piece to the first piece.
\begin{figure}[ht]
\begin{center}
\includegraphics[height=4cm]{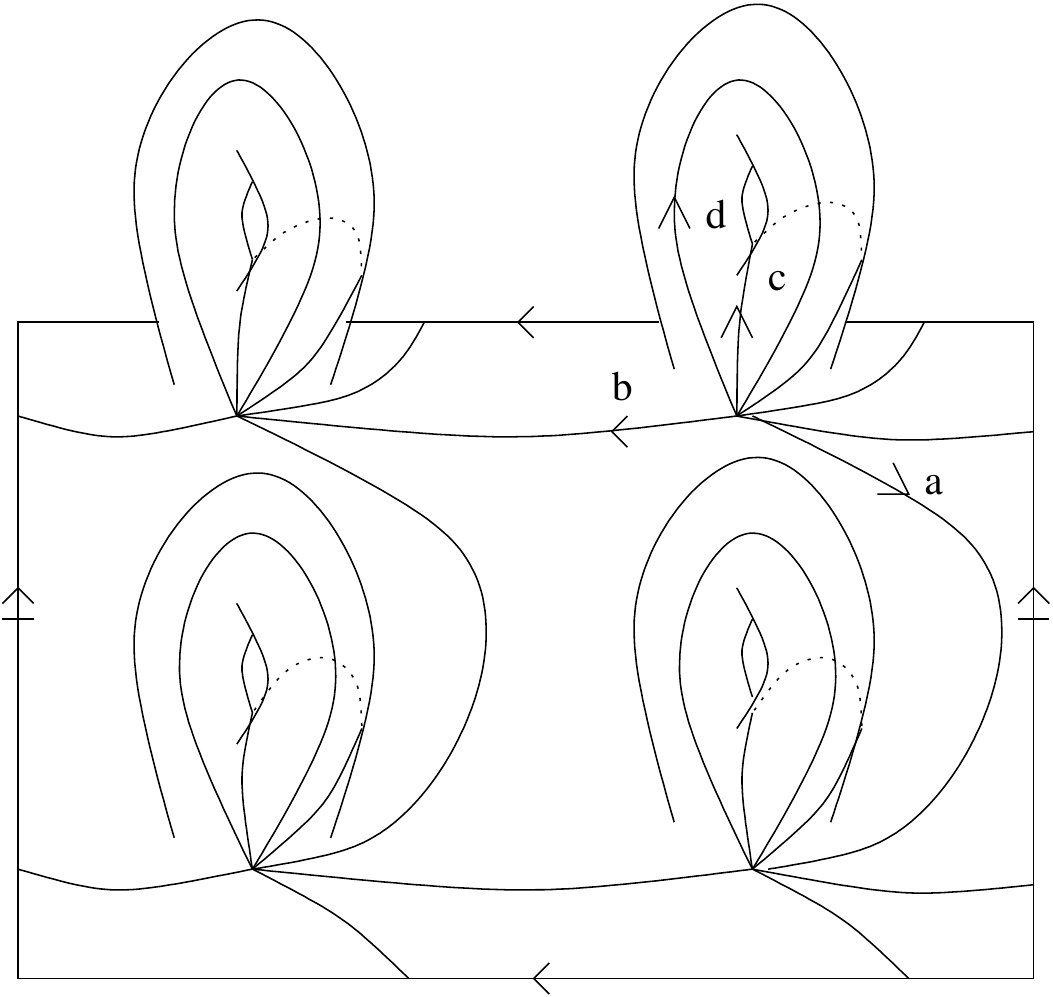} 
\caption{The covering space $M_N$.} 
\label{F:the covering space $M_N$}
\end{center}
\end{figure}
\begin{figure}[ht]
\begin{center}
\includegraphics[height=6cm]{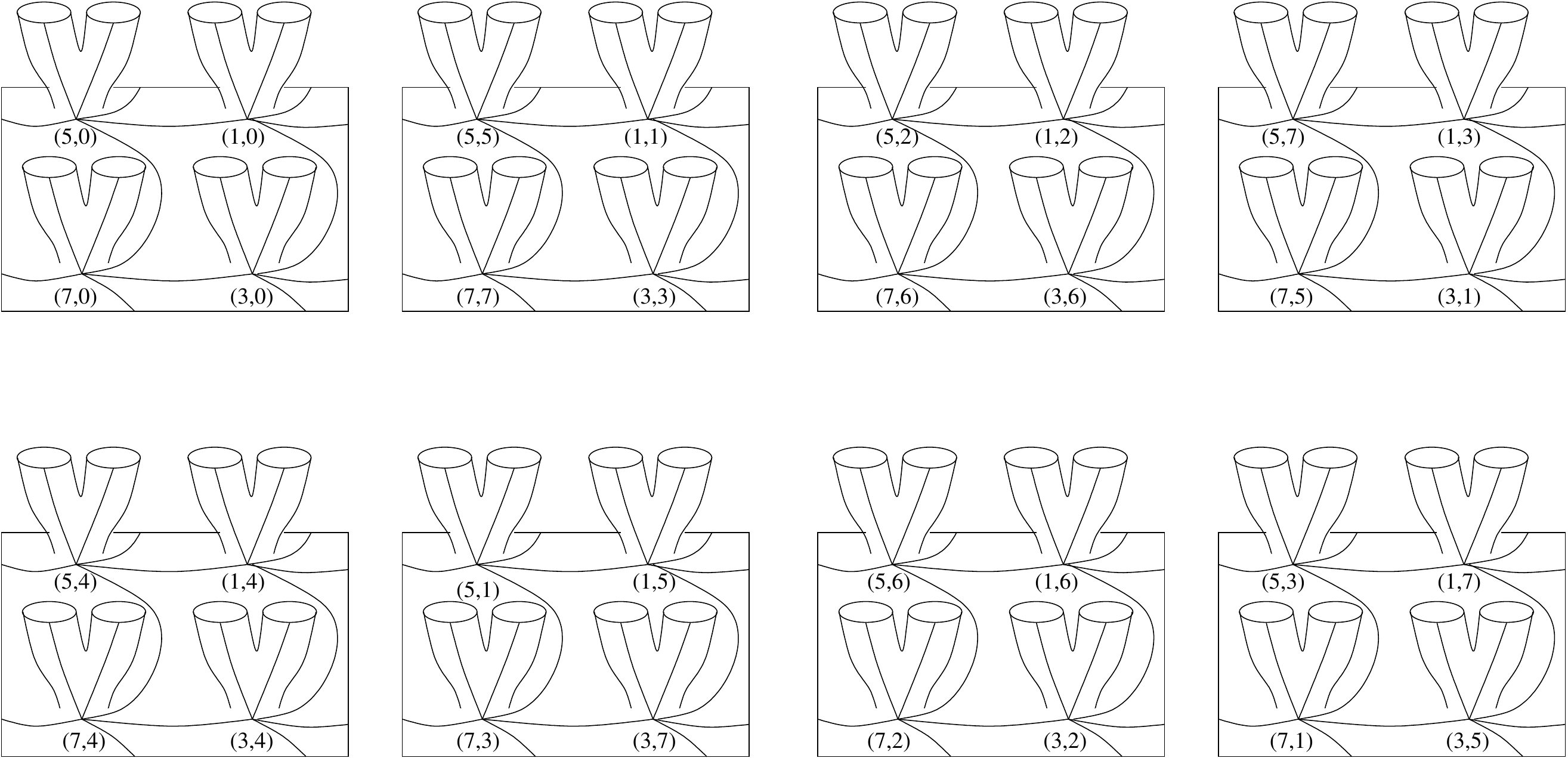} 
\caption{The covering space $M$.} 
\label{F:the covering space $M$}
\end{center}
\end{figure}
\begin{lem}
\label{L:beta}
Let $\alpha = abd[d,c^{-1}]d^{-1}$ be a closed curve on $M_0$. Then $\pi_{H}^{-1}(\alpha)=\beta_{1}^{H}\cup\dots\cup\beta_{5}^{H}$, $\pi_{K}^{-1}(\alpha)=\beta_{1}^{K}\cup\dots\cup\beta_{5}^{K}$ where $\pi_{H}|_{\beta_{i}^{H}}$, $\pi_{K}|_{\beta_{i}^{K}}$ are degree one, for $i=1,2$, and degree two, for $i=3, 4, 5$. Furthurmore $\beta_{1}^{H}$, $\beta_{2}^{H}$ are nonsimple and $\beta_{1}^{K}$, $\beta_{2}^{K}$ are simple.
\end{lem}
\begin{figure}[ht]
\begin{center}
\includegraphics[height=2.5cm]{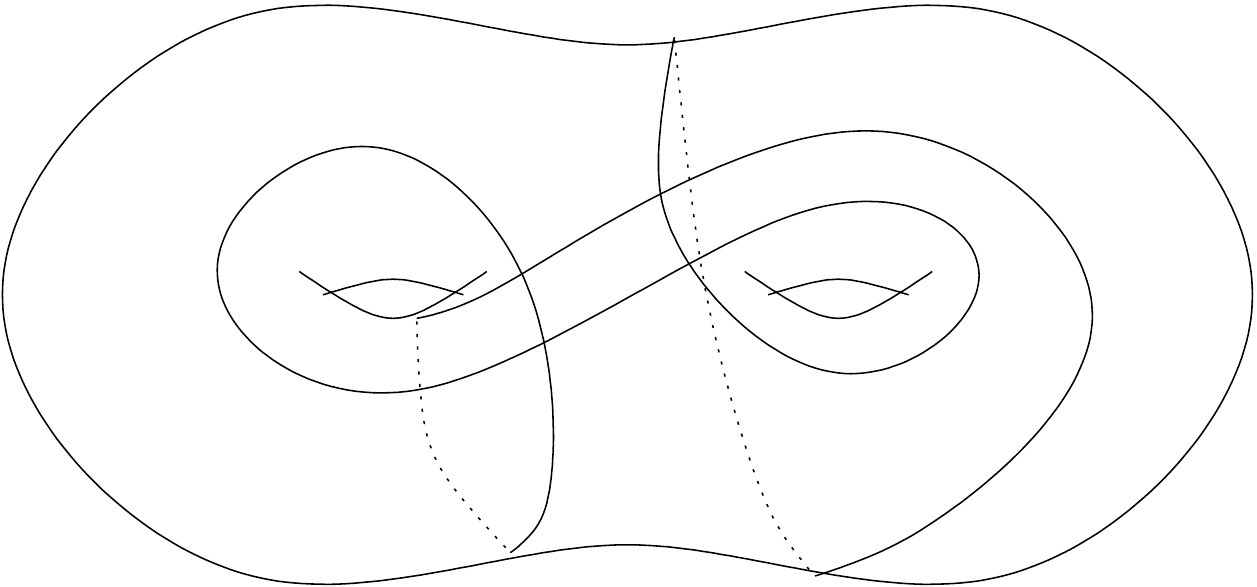} 
\caption{The closed curve $\alpha$ on $M_0$.} 
\label{F:Curve alpha}
\end{center}
\end{figure}
\begin{figure}[ht]
\begin{center}
\includegraphics[height=6cm]{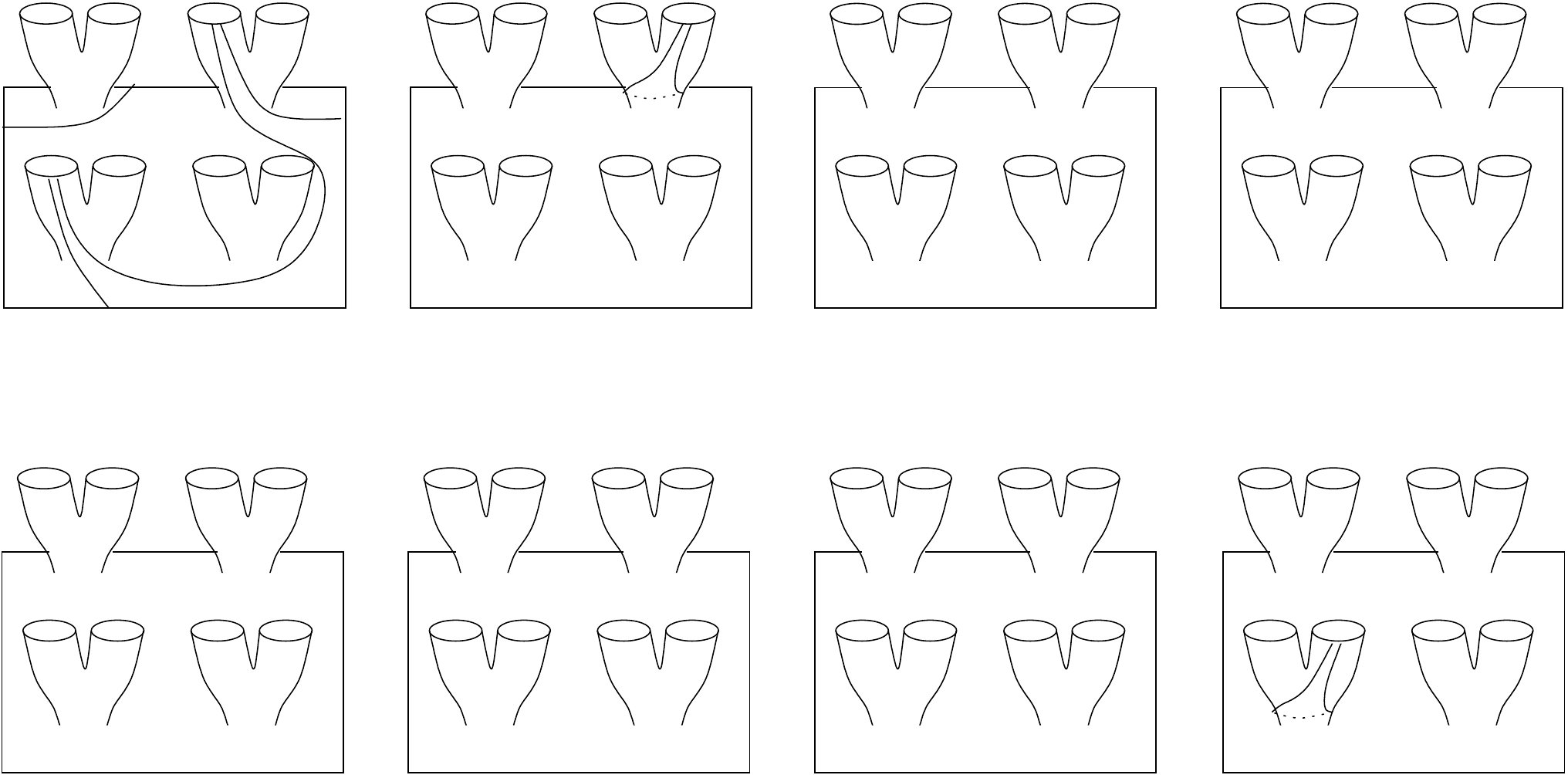} 
\caption{The covering space $M$ and a component $\gamma_{1}$ of $\pi^{-1}(\alpha)$.} 
\label{F:the covering space $M$ and alpha}
\end{center}
\end{figure}
\begin{proof}
First we look at a component $\gamma_1$ of $\pi^{-1}\left(\alpha\right)$ in $M$, see Figure~\ref{F:the covering space $M$ and alpha}. Observe that the preimage of  $\alpha$ is sixteen simple closed curves on M denotes $X=\left\{\gamma_{1},\dots,\gamma_{16}\right\}$. $G$ acts on $X$ and this action is equivalent to the action of $G$ on the cosets of $L=\Stab_{G}(\gamma_{1})=\left\{(1,0), (7,0)\right\}$. More precisely, the bijection 
\[G/\!/L \to X\]
given by 
\[gL \mapsto g\cdot\gamma_{1}\]
is equivariant with respect to the actions of $G$.
We assume $\left\{\gamma_{1},\dots,\gamma_{16}\right\}$ are numbered so that 
\[
	\begin{array}{rrrr}
	\gamma_{1} \rightarrow L, 			& \gamma_{2}\rightarrow (1,1)L, 	& \gamma_{3}\rightarrow(1,2)L, 	& \gamma_{4}\rightarrow(1,3)L,\\
	\gamma_{5} \rightarrow (1,4)L,	& \gamma_{6}\rightarrow (1,5)L, 	& \gamma_{7}\rightarrow(1,6)L, 	& \gamma_{8}\rightarrow(1,7)L,\\
	\gamma_{9} \rightarrow (3,0)L,	& \gamma_{10}\rightarrow (3,3)L,	& \gamma_{11}\rightarrow(3,6)L,	& \gamma_{12}\rightarrow(3,1)L,\\
	\gamma_{13}\rightarrow (3,4)L,	& \gamma_{14}\rightarrow (3,7)L,	& \gamma_{15}\rightarrow(3,2)L,	& \gamma_{16}\rightarrow(3,5)L.
	\end{array}
\]
We use the above representations to compute $H$ and $K$ orbits under the actions of $H$ and $K$ on $X$. Then the $H$ orbits partition $\left\{\gamma_{1},\dots,\gamma_{16}\right\}$ as
\[
\left\{\gamma_{1},\gamma_{9}\right\},
\left\{\gamma_{5},\gamma_{13}\right\},
\left\{\gamma_{2},\gamma_{8},\gamma_{10},\gamma_{16}\right\},
\left\{\gamma_{3},\gamma_{7},\gamma_{11},\gamma_{15}\right\},
\left\{\gamma_{4},\gamma_{6},\gamma_{12},\gamma_{14}\right\} 
\] 
and the $K$ orbits partition $\left\{\gamma_{1},\dots,\gamma_{16}\right\}$ as
\[
\left\{\gamma_{1},\gamma_{13}\right\},
\left\{\gamma_{5},\gamma_{9}\right\},
\left\{\gamma_{2},\gamma_{8},\gamma_{10},\gamma_{14}\right\},
\left\{\gamma_{3},\gamma_{7},\gamma_{11},\gamma_{15}\right\},
\left\{\gamma_{4},\gamma_{6},\gamma_{10},\gamma_{16}\right\}. 
\]
All closed curves in each $H$ orbit lie above exactly one closed curve on $M_{H}$ and all closed curves in each $K$ orbit lie above exactly one closed curve on $M_{K}$. So we can write $\pi_{H}^{-1}(\alpha)=\beta_{1}^{H}\cup\dots\cup\beta_{5}^{H}$ and $\pi_{K}^{-1}(\alpha)=\beta_{1}^{K}\cup\dots\cup\beta_{5}^{K}$. We may associate $\beta_{1}^{H}$, $\beta_{2}^{H}$, $\beta_{1}^{K}$ and $\beta_{2}^{K}$ with the orbits $\left\{\gamma_{1}, \gamma_{9}\right\}$, $\left\{\gamma_{5}, \gamma_{13}\right\}$, $\left\{\gamma_{1}, \gamma_{13}\right\}$ and $\left\{\gamma_{5}, \gamma_{9}\right\}$, respectively.

Next we observe that $\pi_{H}|_{\beta_{i}^{H}}$, $\pi_{K}|_{\beta_{i}^{K}}$ are degree one, for $i=1,2$, and degree two, for $i=3, 4, 5$.

For the simplicity of $\beta_{1}^{H}$, $\beta_{2}^{H}$, $\beta_{1}^{K}$ and $\beta_{2}^{K}$, we look at their associated orbits. We observe that $\gamma_{1}$ intersects $\gamma_{9}=(3,0)\cdot\gamma_{1}$ nontrivially by inspecting Figure~\ref{F:the covering space $M$} for the actions of $G$ and Figure~\ref{F:the covering space $M$ and alpha} for the picture of $\gamma_{1}$. Similarly we can compute
\[
\begin{array}{rr}
	\gamma_{1}\cap \gamma_{9\ }	\neq \emptyset,	&\gamma_{5}\cap \gamma_{13} \neq \emptyset, \\
	\gamma_{1}\cap \gamma_{13}	=    \emptyset,	&\gamma_{5}\cap \gamma_{9\ } =    \emptyset.
\end{array}
\]
Since the $H$ orbit $\left\{\gamma_{1}, \gamma_{9}\right\}$ corresponding to $\beta_{1}^{H}$ contains intersecting curves, $\beta_{1}^{H}$ is nonsimple. Similarly, $\beta_{2}^{H}$ is also nonsimple. Since the $K$ orbit $\left\{\gamma_{1}, \gamma_{13}\right\}$ corresponding to $\beta_{1}^{K}$ contains pairwise disjoint curves, $\beta_{1}^{K}$ is simple. Similarly, $\beta_{2}^{H}$ is also simple.
\end{proof}

To prove Theorem~\ref{T:main theorem}, we will show that generically a hyperbolic metric on $M_0$ lifted to a hyperbolic metric on $M_H$ has the property that there are exactly four closed curves on $M_H$ having the same length as $\beta_{1}^{H}$(and $\beta_{2}^{H}$) and these four closed curves are nonsimple. In the previous Lemma, we found two such closed curves, namely $\beta_{1}^{H}$ and $\beta_{2}^{H}$. Lemma~\ref{L:tau} provides the other two closed curves and we will use Lemma~\ref{L:trace check} to show that there are exactly four such closed curves. Since $M_K$ has a simple closed curve, $\beta_{1}^{K}$, of the same length in its lifted metric, $M_H$ and $M_K$ cannot be simple iso-length spectral.

Let $\tau: M_0 \to M_0$ be the hyperelliptic involution. $\tau$ is isotopic to an isometry for any hyperbolic metric on $M_{0}$. So for any curve $\lambda$ on $M_0$, $\length_{M_0}\left(\lambda\right)=\length_{M_0}\left(\tau\left(\lambda\right)\right)$. For a specific basepoint, the induced map $\tau_\ast:\pi_{1}(M_0) \to \pi_{1}(M_0)$ can be computed to be 
\[
\begin{array}{ll}
\tau_\ast\left(a\right)= a^{-1}, &\tau_\ast\left(b\right)= b^{-1},\\
\tau_\ast\left(c\right)= ac^{-1}dc^{-1}d^{-1}ca^{-1}, &\tau_\ast\left(d\right)= b^{-1}ad^{-1}ba^{-1}.
\end{array}
\]
We have the following lemma.

\begin{lem}
\label{L:tau}
The hyperelliptic involution $\tau: M_0 \to M_0$ lifts to $\tau_H: M_H \to M_H$ and $\tau_K: M_K \to M_K$. In particular, $\tau_H\left(\beta_{i}^H\right)\subset M_H$ is nonsimple and $\tau_K\left(\beta_{i}^K\right)\subset M_K$ is simple, for $i=1,2$.
\end{lem}

\begin{proof}
Let $\psi: G \to G$ be the automorphism of $G$ defined by $\psi(j,k)=(j,-k)$, for any element $(j,k) \in G$. Then we can compute $\psi \circ \rho = \rho \circ \tau_{\ast}$ and $H = \psi^{-1}(H)$. So $\rho^{-1}(H)=\rho^{-1}(\psi^{-1}(H))= \tau_{\ast}^{-1}(\rho^{-1}(H))$. Thus 
\[
\tau_{\ast}\left(\left(\pi_{H}\right)_{\ast}\left(\pi_{1}\left(M_{H}\right)\right)\right)
=\tau_{\ast}\left(\rho^{-1}\left(H\right)\right)
=\rho^{-1}\left(H\right)
=\left(\pi_{H}\right)_{\ast}\left(\pi_{1}\left(M_{H}\right)\right).
\]
Hence the lifting criterion implies that we may lift $\tau$ to $\tau_H$. The existance of a lift $\tau_K$ to $M_K$ is proven in the same way. 
\end{proof}

\begin{lem}
\label{L:trace check}
For almost every $[m] \in \T(M_0)$, if $\gamma$ is a closed curve, $k \in \mathbb Q$ and
\begin{center}
$k \cdot \length_{[m]}\left(\gamma\right)=\length_{[m]}\left(\alpha\right)$
\end{center}
then $k=1$ and $\gamma=\alpha$ or $\tau\left(\alpha\right)$.
\end{lem}

\begin{proof}
For any $\gamma$ and any $k$, either $k\cdot\length_{[m]}(\gamma)=\length_{[m]}(\alpha)$ is true for every $[m]$ or $k\cdot\length_{[m]}(\gamma)\neq \length_{[m]}(\alpha)$ for almost every $[m]$, by Theorem~\ref{T:length_function}. So it suffices to show that if $k\cdot\length_{[m]}(\gamma)=\length_{[m]}(\alpha)$, for every $[m]$, then $k=1$ and $\gamma=\alpha$ or $\tau\left(\alpha\right)$.
\begin{figure}[ht]
\begin{center}
\includegraphics[height=2.5cm]{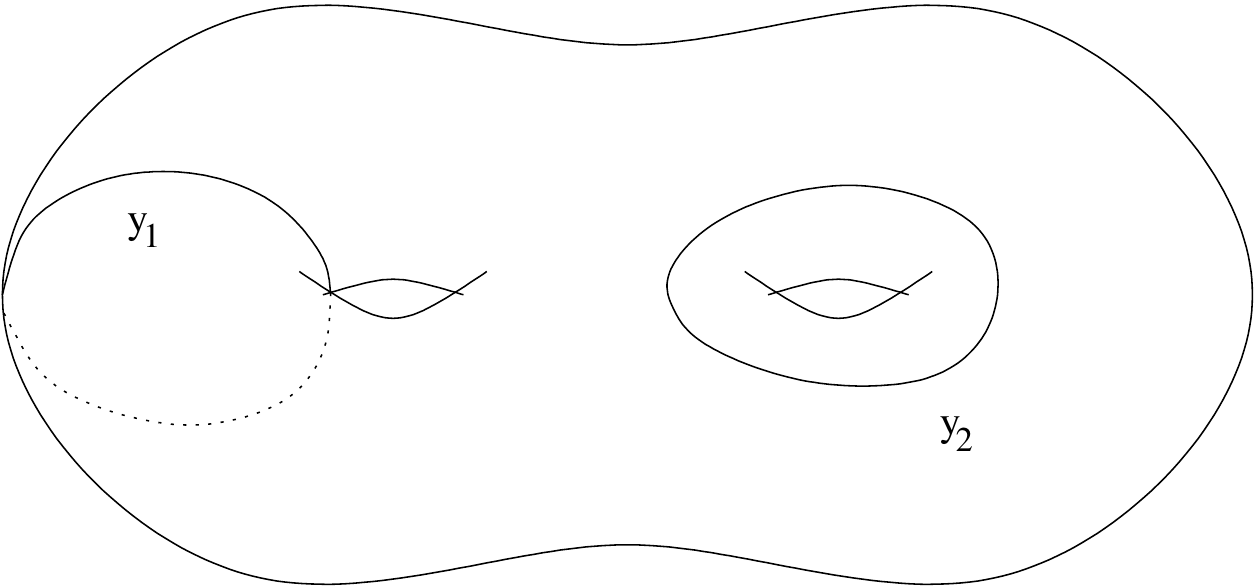} 
\caption{The simple closed curves $x_1$ and $x_2$ on the surface $M_0$.} 
\label{F:x}
\end{center}
\end{figure}

Let $y_1$ be a simple closed curve as shown in Figure~\ref{F:x}. The geometric intersection number of $\alpha$ and $y_1$ is $i(\alpha,y_1)=1$. Since $k\cdot\length_{[m]}(\gamma)=\length_{[m]}(\alpha)$, by Theorem~\ref{T:geometric_intersection}, $k \cdot i(\gamma,y_1)=i(\alpha,y_1)=1$. Since the geometric intersection numbers are nonnegative intergers, $k=1$.
To prove that $\gamma=\alpha$ or $\tau(\alpha)$, we find some neccessary conditions for $\gamma$ to have the same length as $\alpha$, for every $[m]\in\T(M_0)$

Let $y_2$ be the simple closed curve shown in Figure~\ref{F:x}. Since $i(\gamma,y_2)=i(\alpha,y_2)=0$ by Theorem~\ref{T:geometric_intersection}, $\gamma$ and $\alpha$ are contained in $M_0-y_2$.

We cut $M_0$ along the simple closed curve $y_2$ to get a torus with two holes and change the basis $\left\{a,b,d\right\}$ to the basis $\left\{a,b,x=da^{-1}\right\}$, see Figure~\ref{F:2holes}. Then $\alpha=abxaba^{-1}b^{-1}x^{-1}$ and $\tau_\ast(\alpha)=a^{-1}b^{-1}b^{-1}x^{-1}ba^{-1}b^{-1}axb$.
\begin{figure}[ht]
\begin{center}
\includegraphics[height=3cm]{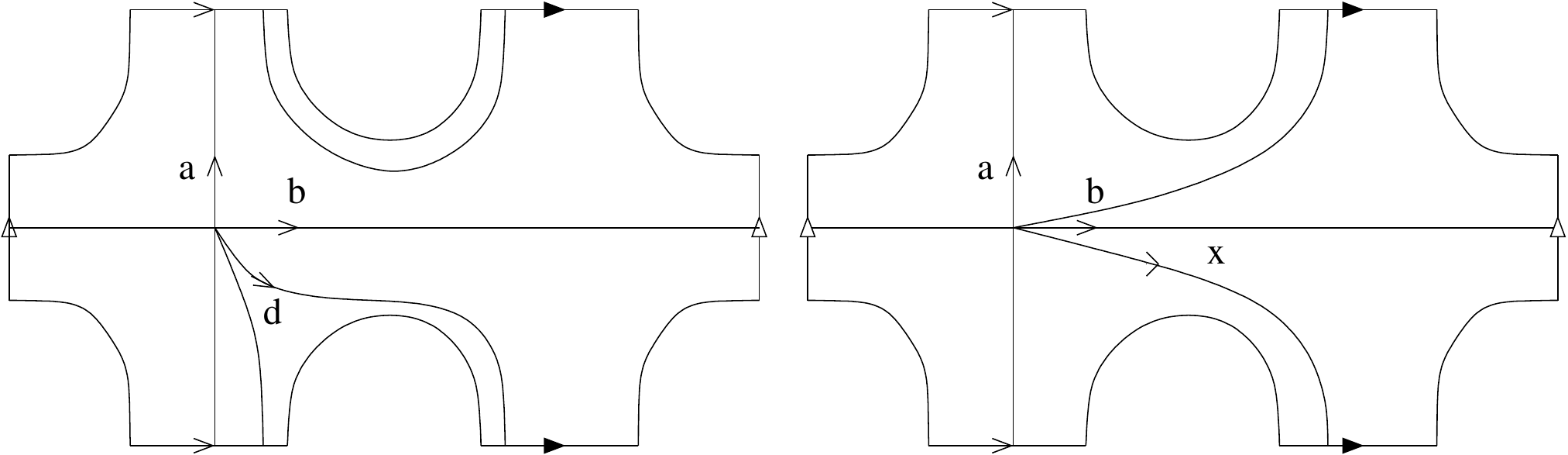} 
\caption{The torus with two holes, $M_0-x_2$.} 
\label{F:2holes}
\end{center}
\end{figure}
Consider the spine as shown in Figure~\ref{F:spine}, we homotope $\alpha$ and $\gamma$ into spine, as edge loops without backtracking. Then by considering metrics on $M_0$ where length of some of the edges are bounded and others tend to infinity, we see that in order for $\gamma$ to have the same length as $\alpha$ in $M_0$,
\begin{center}
$\sharp \left\{a_1 \:\mbox{edges of}\:\gamma\right\} = \sharp \left\{a_1 \:\mbox{edges of}\: \alpha \right\} = 3$,\\
$\sharp \left\{x_1 \:\mbox{edges of}\:\gamma\right\} = \sharp \left\{x_1 \:\mbox{edges of}\: \alpha \right\} = 3$,\\
$\sharp \left\{b_1 \:\mbox{edges of}\: \gamma\right\} + \sharp \left\{b_2 \:\mbox{edges of}\: \gamma\right\} = \sharp \left\{b_1 \:\mbox{edges of}\: \alpha\right\} + \sharp \left\{b_2 \:\mbox{edges of}\:\alpha\right\} =8$.
\end{center}
\begin{figure}[ht]
\begin{center}
\includegraphics[height=3cm]{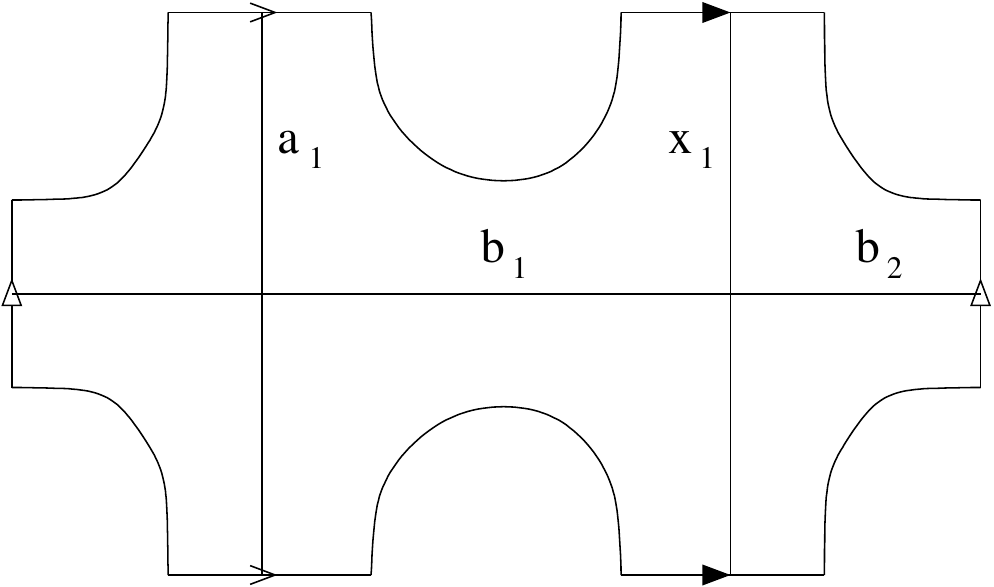} 
\caption{The torus with two holes, $M_0-x_2$ with spine.} 
\label{F:spine}
\end{center}
\end{figure}
Since $\length_{[m]}(\gamma)=\length_{[m]}(\alpha)$ and $\left[\alpha\right]=\left[ab\right]\in H_1(M_0)$, $\left[\gamma\right]=\pm\left[ab\right]\in H_1(M_0)$, by Theorem~\ref{T:homology}. Thus from the observation of the edge counts above (replacing $\gamma$ with $\gamma^{-1}$ if necessary), we have the following conditions;
\begin{enumerate}
	\item $\gamma$ consists of exactly two $a$'s, one $a^{-1}$, one $x$, and one $x^{-1}$,
	\item $\sharp \left\{b^{-1}\mbox{'s in}\:\gamma\right\} = \sharp \left\{b \mbox{'s in}\:\gamma\right\}  -1$, and
	\item $\sharp \left\{b_1 \:\mbox{edges of}\:\gamma\right\} + \sharp \left\{b_2 \:\mbox{edges of}\: \gamma\right\} = 8$.
\end{enumerate}

Next we find all closed curves on $M_0$ satisfying these three conditions. By the conditions above we know the exact number of $a$'s, $a^{-1}$'s, $x$'s, and $x^{-1}$ that appear in $\gamma$. So we only need to determine the possible number of $b$'s and $b^{-1}$. To do this, we note that while the number $a_1$-edge and the number of $x_1$-edge can be computed directly by counting the number of $\left\{a, a^{-1}\right\}$ and $\left\{x, x^{-1}\right\}$, respectively, some combinations of $x$'s and $b$'s provide cancellations in the sum of $b_1$ and $b_2$-edge count. One example is that $x$ alone contributes $2$ to the sum of $b_1$ and $b_2$-edge count, $b$ alone also contributes $2$ to the sum of $b_1$ and $b_2$-edge count but $xb$ contributes only $2$ to the sum of $b_1$ and $b_2$-edge count.





Taking this type of cancellation into consideration, we can produce a list $A$ of $4320$ words in $\left\{a^{\pm1},b^{\pm1},x^{\pm1}\right\}$ that contains all curves satisfying the three conditions.

One can explicitly construct $[m] \in \T(M_0)$, a hyperbolic metric on $M_0$ such that
 
\[ \rho_m(a)=  \left( \begin{array}{cc}
5/3 & 3/4  \\
3/4 & 5/4  
\end{array} \right),\]
\[ \rho_m(b)=  \left( \begin{array}{cc}
4 & 0  \\
0 & 1/4  
\end{array} \right),\]
\[ \rho_m(x)=  \left( \begin{array}{cc}
5/3 & -16/3  \\
-1/3 & 5/3  
\end{array} \right).\]

Then the trace of $\rho_m(\alpha)$ is 
\[ 
\tr(\rho_m(\alpha))= 109505/2048.
\]
By using Mathematica, we have that the elements in $A$ having the same trace squared as $\alpha$ are $\alpha$ and $\tau(\alpha)^{-1}$.

So, by equation~(\ref{eq:length_function}), the only curves in $A$ that have the same length in $M_0$ as $\alpha$ are $\alpha$ and $\tau(\alpha)$.

Thus if $\length_{[m]}\left(\gamma\right)=\length_{[m]}\left(\alpha\right)$, for every $[m] \in \T(M_0)$, then $\gamma=\alpha$ or $\tau\left(\alpha\right)$.
\end{proof}

\begin{proof}[Proof of Theorem ~\ref{T:main theorem}]
Let $\rho:\pi_{1}(M_0) \to G$ be the surjective homomorphism defined in this section.

Let $\alpha = abd[d,c^{-1}]d^{-1}$ be a closed geodesic on $M_0$.

By Lemma~\ref{L:beta} and Lemma~\ref{L:tau}, for almost every $[m] \in \T(M_0)$, there are four nonsimple closed geodesics
$ 
\left\{
\beta_{1}^{H}, 
\beta_{2}^{H}, 
\tau_H\left(\beta_{1}^H\right), 
\tau_H\left(\beta_{2}^H\right) 
\right\}
$
on $M_H$ having length $l=\length_{[m]}(\beta_{1}^{H})=\length_{[m]}(\alpha)$ and there are four simple closed geodesics 
$
\left\{
\beta_{1}^{K}, 
\beta_{2}^{K}, 
\tau_K\left(\beta_{1}^K\right), 
\tau_K\left(\beta_{2}^K\right) 
\right\}
$
on $M_K$ having length $l$.

If $\gamma^H$ is a closed geodesic on $M_H$ having length 
\[
l=\length_{[m]}(\beta_{1}^{H})=\length_{[m]}(\alpha),
\]
then $\pi_H(\gamma^H)$ is a closed geodesic on $M_0$ having length 
\[
k\cdot l=k\cdot\length_{[m]}(\beta_{1}^{H})=k\cdot\length_{[m]}(\gamma),
\]
for some $k = 1, 1/2, 1/4, \mbox{or}\;1/8$, since the degree of $\pi_H$ and $\pi_K$ is $8$.

By Lemma~\ref{L:trace check}, $k=1$ and $\pi_H(\gamma^H)=\alpha$ or $\tau(\alpha)$. Thus $\gamma^H$ is one of the four nonsimple closed curves above. Hence there are exactly four closed curves on $M_H$ having length $l$ and those four closed curves are nonsimple. Similarly, there are exactly four closed curves on $M_K$ having length $l$ and those four closed curves are simple.

Therefore $M_H$ and $M_K$ are not simple iso-length spectral.
\end{proof}

\begin{proof}[Proof of Corollary~\ref{C:cor}]  
As the proof of Theorem~\ref{T:main theorem} shows, for almost every $[m] \in \T(M_0)$, there is a simple closed geodesic on $M_K$ with the same length as $\alpha$ on $M_0$, but no such simple geodesic on $M_H$. Therefore, $M_H$ and $M_K$ are not simple length equivalent.
\end{proof}
\section{Final discussion}
\label{sec:Final discussion}
Theorem~\ref{T:main theorem} should hold for any surjective homomorphism $\rho: \pi_1(M_0) \to G$ and for any closed surface $M_0$.  Indeed, it can be shown that for $G$ as in Theorem~\ref{T:main theorem} and any $\rho$, there is a genus $2$ or $3$ subsurface $\Sigma \subset M_0$ so that the restriction $\rho|_{\pi_1(\Sigma)}$ is surjective.  Then, one can list all such surjective homomorphisms and try to construct a curve $\alpha$ in $\Sigma$ playing the role of $\alpha$ in the proof of Theorem~\ref{T:main theorem}.  This does not seem to provide much new information, and even for the cases analyzed by the author, the resulting presentation is significantly more complicated.
It would be interesting to find an approach that works for all homomorphisms simultaneously.

Another class of examples that would be interesting to analyze with respect to Question~\ref{Q:question2} are those given in \cite{Brooks} and \cite{Buser}, as the proof that the surfaces are iso-length spectral is more directly geometric.
\section*{Acknowledgements}
I would like to thank my advisor Christopher J. Leininger for guidance and useful conversations.

\bibliographystyle{plain}
\bibliography{paper}

Department of Mathematics, University of Illinois at Urbana-Champaign, 1409 W. Green st., Urbana, IL 61801. \texttt{e-mail:rmaungc2@illinois.edu}

\end{document}